\documentclass[12pt]{article}
\usepackage{geometry}                
\usepackage{graphicx}
\usepackage{amssymb, enumerate}
\usepackage{epstopdf}
\usepackage{amsmath}
\usepackage{amssymb}
\usepackage{amsthm}
\newtheorem{theorem}{Theorem}
\newtheorem{lemma}[theorem]{Lemma}
\newtheorem{corollary}[theorem]{Corollary}

\newtheorem{question}[theorem]{Question}

\title{On the Ramsey number of the double star}

\author{Freddy Flores Dub\'o and Maya Stein\thanks{Department of Mathematical Engineering, University of Chile, and Center for Mathematical Modeling. MS acknowledges support by ANID Regular Grant 1221905 and by ANID PIA CMM FB210005.}
} 

\date{} 

\begin{document}

\maketitle
\begin{abstract} 
The double star $S(m_1, m_2)$ is obtained from joining the centres of a star with $m_1$ leaves and a star with $m_2$ leaves. 
We give a short proof of a new upper bound on the two-colour Ramsey number of $S(m_1, m_2)$ which holds for all  $m_1, m_2$ with $\frac{\sqrt 5 +1}{2} m_2<m_1<3m_2$. Our result implies that for all positive~$m$, the Ramsey number of the double star $S(2m, m)$  is at most $\lceil 4.275m\rceil+1$.
 \end{abstract}

\section{Introduction}

The much studied {Ramsey number} $R(H)$  of a graph $H$ is defined as the smallest integer $n$ such that every 2-colouring of the edges of~$K_n$ contains a monochromatic copy of $H$. The case when $H$ is a complete graph is the subject of Ramsey's famous theorem from the 1930's, and determining Ramsey numbers of complete graphs is notoriously difficult. For a recent breakthrough, see~\cite{cgms}.

Among the earliest non-complete graphs $H$ to be studied were different kinds of trees. In 1967, Gerencs\'er and Gy\'arf\'as~\cite{GerencserGyarfas} showed that  $R(P_k)=k+\lfloor\frac{k+1}2\rfloor$, where  $P_k$ is  the $k$-edge path. 
For $k$-edge stars $K_{1,k}$, the Ramsey number is  larger: Harary~\cite{Harary:RecentRamsey} observed in 1972 that $R(K_{1,k},K_{1,k})=2k$ if   $k$  is odd, and $R(K_{1,k},K_{1,k})=2k-1$ if $k$ is even. 

Burr and Erd\H os~\cite{BE76} conjectured in 1976 that  $R(T_k)\le R(K_{1,k},K_{1,k})$, for any tree~$T_k$ with $k$ edges. For large $k$, it is known that $R(T_k)\le 2k$, by the results of~\cite{Zhao2011}.
However,   this bound far from best possible for paths, which motivated the search for a more fine-tuned conjecture. Note that paths are (almost) completely balanced trees, while stars  are the most  unbalanced trees. So, it seems natural to suspect that the Ramsey number of a tree might be related to its unbalancedness, i.e.~the difference in size between the two bipartition classes.

It is easy to see that $$R_B(T):=\max\{2t_1, t_1+2t_2\}-1$$ is a lower bound for the Ramsey number of any tree $T$ with bipartition classes of sizes $t_1\ge t_2\ge 2$. This can be seen by considering the {\it canonical colourings}, which are defined as follows. Take a complete graph $G$ on $R_B(T)-1$ vertices.  If $t_1>2t_2$, partition $V(G)$ into two sets of equal size, colour all edges inside each set  red and colour all remaining edges blue. If $t_1\le 2t_2$, take a set of $t_1+t_2-1$ vertices, colour all edges inside this set red, and  colour all remaining edges blue. It is straightforward to see that no monochromatic copy of $T$ is present in this colouring.

Note that if $T$ is a path then $R_B(T)=R(T)$, and the same holds if $T$ is a star with an even number of edges. In~\cite{Burr74}, 
Burr discusses the canonical colourings and expresses his belief that $R(T)$ may be equal to $R_B(T)$ unless $T$ is an odd star. In 2002,
 Haxell, \L uczak, and Tingley~\cite{HLT02} confirmed this suspicion asymptotically for all   trees with linearly bounded maximum degree. Namely, they proved that for every~$\eta>0$, there exist $t_0$ and $\delta$ such that $R(T)\le (1+\eta)R_B(T)$ for each tree $T$ with $\Delta(T)\le\delta t_1$ and $t_1>t_0$, where $t_1\ge t_2$ are, as before, the sizes of the bipartition classes of the tree $T$.
 
 But already in 1979, Grossman, Harary and Klawe~\cite{GHK79} found that, contrary to Burr's suspicion, there are values of $m_1, m_2$ such that $R(S(m_1, m_2))>R_B(S(m_1, m_2))$ (where $S(m_1, m_2)$ is the double star with $m_i$ leaves in partition class~$i$). 
   However, the examples from~\cite{GHK79} still allowed for the possibility that  for every tree $T$ we would have that $R(T)\le R_B(T)+1$. The authors of~\cite{GHK79} conjectured this to be the truth for all double stars, which they confirmed    for a range of values of $m_1, m_2$. Currently, it is known that this holds if $m_1\ge 3m_2$~\cite{GHK79} or  if $m_1\le 1.699(m_2+1)$~\cite{NSZdouble}. In other words, for $m_1, m_2\in\mathbb N^+$ 
  it holds that 
  \begin{equation}\label{burr}
  R(S(m_1, m_2))\le \max\{2m_1, m_1+2m_2\}+2=R_B(S(m_1, m_2))+1
    \end{equation}
  unless 
  \begin{equation}\label{range}
  \text{$1.699(m_2+1)<m_1< 3m_2$.}
  \end{equation}
 
But  in general, inequality~\eqref{burr} is not true.  Norin, Sun and Zhao~\cite{NSZdouble}   showed that $R(S(m_1, m_2))\ge 5m_1/3+5m_2/6+o(m_2)$ for all $m_1\ge m_2\ge 0$ and $R(S(m_1, m_2))\ge 189m_1/115+21m_2/23+o(m_2)$ for all $m_1\ge 2m_2\ge 0$.
In particular, their results imply that   $R(S(m_1, m_2))>R_B(S(m_1, m_2))+1$  if $m_1, m_2$ fulfill
\begin{equation*}\label{nsz}
\frac 74 m_2+o(m_2)\le m_1\le \frac {105}{41} m_2+o(m_2).
\end{equation*}
This range covers the special case that $m_1=2m_2$. For this case, the results from~\cite{NSZdouble} yield that
$R(S(2m, m))\ge4.2m+o(m)$ while $R_B(S(2m, m))=4m+2.$
This discovery lead the authors of~\cite{NSZdouble} to pose the following question.

\begin{question}[Norin, Sun and Zhao~\cite{NSZdouble}]
Is it true that $R(S(2m, m))=4.2m+o(m)$?
\end{question}

There are few results giving upper bounds on the Ramsey number of the double star for the range of $m_1, m_2$ where~\eqref{burr} does not hold. The inequality $R(S(m_1, m_2))\le 2m_1+m_2+2$ for all $m_1\ge m_2\ge 0$ was established in~\cite{GHK79}, where it is described as a `weak upper bound'. In the preprint~\cite{NSZdouble}, very good asymptotic bounds for $R(S(m_1, m_2))$ are obtained from a computer-assisted proof using the flag algebra method, but as these are not quick to state, we refer the reader to~\cite{NSZdouble}. We remark that  Theorem 4.5 from~\cite{NSZdouble}, used with the invalid pair number 5 from Table 1 of~\cite{NSZdouble}, implies that $\lim_{m\to\infty}R(S(2m, m))/m$ is bounded from above by $4.21526$.
 
Our main result is a short elementary proof of a new upper bound on 
$R(S(m_1,m_2))$  which holds for all values of $m_1, m_2\in \mathbb N^+$ fulfilling $\frac{\sqrt 5 +1}{2} m_2<m_1<3m_2$. 
Observe that $\frac{\sqrt 5 +1}{2}>1.618$, and thus our result covers the whole range of values of $m_1, m_2$ from~\eqref{range}.
\begin{theorem}
\label{theorem12}
Let $ m_1, m_2 \in \mathbb N^+$, with $\frac{\sqrt 5 +1}{2} m_2<m_1<3m_2$.
Then
$$ R (S (m_1, m_2)) \leq \Big\lceil\sqrt{2m_1^2+(m_1+\frac{m_2}2)^2} +\frac{m_2}2\Big\rceil+1. $$
\end{theorem}


As an immediate corollary of our theorem, we obtain for the double star $S(2m, m)$ the following bound. 

\begin{corollary}
\label{theorem11}
$ R (S (2m, m)) \leq \lceil4.27492m\rceil + 1$ for all $ m \in \mathbb N^+$.
\end{corollary}

\section{Preliminaries}

In this section we prepare the proof of the main result, Theorem~\ref{theorem12}, by  proving some auxiliary results.
We start with a very simple lemma  for recurrent later use. A similar lemma appears in~\cite{NSZdouble}.
\begin{lemma}
\label{prop1}
Let $m_1, m_2\in\mathbb N$,  let $G$ be a graph and let $vw \in E (G) $ such that $ d (v) > m_1 $, $d(w)> m_2$, and $ | N (v) \cup N (w) | \geq m_1+m_2 + 2 $.
Then
$S (m_1, m_2)  \subseteq G. $
\end{lemma}

\begin{proof} 
To form the double star with central edge $ vw$, first choose $ m_1$ neighbours of~$v $, as many as possible outside $N (w) \cup \{w\}$, the others in $N (w)$. Then,  choose~$m_2 $ neighbours of $ w $ in $ N (w) $, different from $v$ and from the previously  chosen neighbours of $ v $. This concludes the proof.
\end{proof}

Next we show a useful statement about vertex degrees when no double  star is present.

\begin{lemma}
\label{lema2}
Let $m_1, m_2\in\mathbb N$, and let $ G $ be a graph  on $n\ge m_1+m_2+2$ vertices  such that $ S (m_1, m_2)\not \subseteq  G$. 
Let $ v \in V (G) $, let $A\subseteq N(v)$ with $|A|> m_1$ and $d(u)> m_2$ for each $u\in A$.
Let $w\in A$. 
 Then   $w$ has at most $m_1+m_2-|A|$ neighbours in $V(G)\setminus (A\cup\{v\})$. Furthermore,   there is a vertex $z\in V(G)\setminus (A\cup\{v\})$ having at most $$\frac {m_1+m_2-|A|}{n - |A|-1}\cdot |A|$$ neighbours in~$A$.
\end{lemma}

\begin{proof} 
Set $ D: = V(G)\setminus (A\cup\{v\}) $. If $ w $ has  $m_1+m_2-|A|+1$ or more neighbours in~$D$, then $| N (v) \cup N (w) | \ge |A|+  (m_1+m_2-|A|+1) +|\{v\}| = m_1+m_2+2$ (we count $v$ as a neighbour of $w$), and we can apply Lemma $ \ref {prop1} $ to see that $S (m_1, m_2) \subseteq G$, which is  a contradiction. 

So $ w $ has at most $m_1+m_2-|A|$  neighbours in  $ D $, which is as desired. Further, as this holds for every $u\in A$, the average number of 
neighbours in $A$ of 
 a vertex from~$D $ is at most
\begin{align*}\frac {(m_1+m_2-|A|)\cdot |A|} {|D|}
& =
\frac {m_1+m_2-|A|}{n - |A|-1}\cdot |A|.
\end{align*}
So any vertex $z\in D$ having at most the average number of 
neighbours in $A$  is as desired.
\end{proof}

We will also need a lemma from~\cite{NSZdouble}, whose elementary proof can be found there. 
\begin{lemma}[Lemma 2.3 in~\cite{NSZdouble}]\label{colours}
Let $n\ge \max\{2m_1, m_1+2m_2\}+2$, and let the edges of $K_n$ be coloured with red and blue such that there is no monochromatic $S(m_1, m_2)$. Then there is a colour $C\in\{red, blue\}$ such that each vertex of $K_n$ has degree at most $m_1$ in colour $C$.
\end{lemma}

\section{Proof of Theorem \ref {theorem12}. }

The whole section is devoted to the proof of Theorem 
\ref {theorem12}. Let $m_1,m_2\in\mathbb N^+$ be given, fulfilling
\begin{equation}\label{m1m2}
\frac{\sqrt 5 +1}{2} m_2<m_1<3m_2.
\end{equation}
 Set
  \begin{equation}\label{thisism3}
  m_3:= \Big\lceil\sqrt{2m_1^2+(m_1+\frac{m_2}2)^2} - (m_1+\frac{m_2}2)\Big\rceil.
  \end{equation}
Using~\eqref{m1m2} and~\eqref{thisism3}, it is easy to calculate that
\begin{equation}\label{thatisalpha}
m_3>\max\{m_2, m_1-m_2\},
\end{equation}
and in particular, we have that $m_3\ge 1$.
Set 
$n:=m_1+m_2 +m_3+1$,
and let a red and blue colouring of the edges of $K_n$  be given. 
Let~$G_r$ be the subgraph of $K_n$ induced by the red edges, and $G_b$ be the subgraph of $K_n$  induced by the blue edges. For any~$u\in V(K_n)$, let $N_r(u)$ be the set of all neighbours of  $u$ in $G_r$, and let $N_b(u)$ be the set of all neighbours of  $u$ in $G_b$. Set $d_r(u):=|N_r(u)|$ and $d_b(u):=|N_b(u)|$.

For contradiction assume that there is no monochromatic $S (m_1, m_2)$. Note that   $n\ge  \max\{2m_1, m_1+2m_2\}+2$ because of~\eqref{thatisalpha} 
and since $n$ is an integer. So, we can use Lemma~\ref{colours} to see that there is a colour, which we may assume to be blue, such that 
every vertex has degree at most $m_1$ in that colour. That is, $d_b(u)\le m_1$  for all $u\in V(G)$, and thus, 
\begin{equation}\label{mindeg} 
\delta(G_r) \ge  m_2+m_3.
\end{equation}

Now choose any vertex $v$ and a subset $A$ of $N_r(v)$ with 
\begin{equation}\label{A} 
|A|=m_2+m_3.
\end{equation}

By~\eqref{mindeg}, and since $m_2+m_3>m_1$ by~\eqref{thatisalpha}, we know that $|A|>m_1$ and $\delta(G_r) >  m_2$. So, we can use
  Lemma \ref {lema2} in $G_r$ 
   to see that for any $w\in A$, we have
\begin{equation*}\label{wew}
|N_r(w)\setminus (A\cup \{v\})|\le  m_1+m_2-(m_2+m_3)=m_1-m_3. 
\end{equation*}
and therefore, 
\begin{align}\label{wnew}
|N_r(w)\cap  (A\cup \{v\})| & = 
d_r(w)-|N_r(w)\setminus (A\cup \{v\})| \notag \\ &\ge
m_2+m_3-(m_1-m_3)
\notag \\ &=m_2+2m_3-m_1.
\end{align}
We employ
 Lemma \ref {lema2} once more, this time to find a vertex $z\notin A\cup\{v\}$  such that
 \begin{equation*}\label{xxx}
|N_r(z)\cap A| \le  \frac {m_1+m_2-|A|}{n - |A|-1}\cdot |A|=  \frac { m_1-m_3}{m_1}\cdot   (m_2+m_3),
   \end{equation*}
    where we use~\eqref{A} for the equality.
We deduce that
\begin{align}
|N_r(z)\setminus A| & = d_r(z)-|N_r(z)\cap A|\notag \\ & \ge  (m_2+m_3)- \frac { m_1-m_3}{m_1}\cdot   (m_2+m_3) 
\notag \\ &=
 (m_2+m_3)\frac{m_3}{m_1}.\label{nownow}
\end{align}

Further, note that $d_b(z)\le m_1<m_2+m_3=|A|$ because of~\eqref{mindeg}, \eqref{thatisalpha} and~\eqref{A}. Therefore, we know that vertex $z$ sends at least one red edge to~$A$.
Consider any red edge $uz$ with $u\in A$. 
Using~\eqref{wnew} and~\eqref{nownow},  we get
\begin{align*}
|N_r(u)\cup N_r(z)|
 & \ge |N_r(u)\cap (A\cup\{v\})|+|N_r(z)\setminus A|+|\{u,z\}|\\
& \ge m_2+2m_3-m_1+    (m_2+m_3)\frac{m_3}{m_1}+2\\
& \ge m_1+m_2+2,
\end{align*}
where for the last inequality we use the fact that $2m_1m_3+    m_2m_3+m_3^2\ge 2m_1^2$ 
which can be calculated from~\eqref{thisism3}. 
So, we can apply Lemma~\ref {prop1} to find a red double star with central edge $uz$, and are done.

\section*{Acknowledgment}
The second author would like to thank C. Karamchedu, M. Karamchedu and M.~Klawe for pointing out a missing ceiling in the bound of Theorem~\ref{theorem12} in an earlier version of this paper.

\bibliographystyle{acm}
\bibliography{trees}

\end{document}